\documentclass[12pt]{amsart}

\usepackage{enumerate}
\usepackage[margin=2.7cm]{geometry}
\usepackage{graphics, amsmath,amssymb, enumitem, comment, url, mathtools}
\usepackage{graphicx}
\usepackage{epsfig}

\newif\ifcolorcomments
\newcommand{\allowcomments}[4]{
\newcommand{#1}[1]{\ifdraft{\ifcolorcomments{\textcolor{#4}{##1 --#3}}\else{\textsl{ ##1 \ --#3}}\fi}\else{}\fi}
}

\colorcommentstrue
\usepackage{amssymb}
\usepackage{amsmath,amsthm}

\usepackage{colortbl}

\newtheorem{theorem}{Theorem}[section]
\newtheorem{lemma}[theorem]{Lemma}

\theoremstyle{definition}

\newcommand{\calA}{\mathcal{A}}

\newcommand{\bfA}{\mathbf A}

\newcommand{\E}{\mathbb E}

\newcommand{\N}{\mathbb N}% by default, $\N$ is American naturals $\{1,2,\ldots\}$
\newcommand{\R}{\mathbb R}

\newcommand{\bfa}{\mathbf{a}}

\newcommand{\bfu}{\mathbf{u}}

\newcommand{\bff}{\mathbf{f}}
\newcommand{\bfl}{\mathbf{l}}

\newcommand{\veca}{\overline{a}}
\newcommand{\vecb}{\overline{b}}
\newcommand{\vecc}{\overline{c}}
\newcommand{\vecd}{\overline{d}}

\DeclareMathOperator{\Des}{Kid}

\renewcommand{\text}{\textup}

\newcommand{\NPC}[1]{\ignorespaces}

\newif\ifdraft\drafttrue

\def\N{\mathbb N}

\def\R{\mathbb R}

\newcommand\hdim{\dim_{\mathrm H}}

\IfFileExists{marvosym.sty}{
\RequirePackage{marvosym} 
}{

}
\let\origemptyset\emptyset
\IfFileExists{wasysym.sty}{
\RequirePackage{wasysym}\renewcommand{\emptyset}{{\diameter}}
}{

}

\mathtoolsset{showonlyrefs}
\newcommand*{\myDots}{\ifmmode\mathellipsis\else.\kern-0.07em.\kern-0.07em.\fi}

\usepackage[dvipsnames]{xcolor}

\allowcomments{\commumtaz}{MH}{Mumtaz}{green}
\allowcomments{\comnikita}{NS}{Nikita}{blue}
\allowcomments{\comgero}{GGR}{Gero}{red}

\newcommand {\ignore}[1] {}

\begin{document}

\title[Restricted slowly growing digits in iIFS]{Restricted slowly growing digits for infinite iterated function systems}

\author[G. Gonz\'alez Robert]{Gerardo Gonz\'alez Robert}
\address{Gerardo Gonz\'alez Robert, Department of Mathematical and Physical Sciences,  La Trobe University, Bendigo 3552, Australia. }
\email{G.Robert@latrobe.edu.au}

\author[M. Hussain]{Mumtaz Hussain}
\address{Mumtaz Hussain,  Department of Mathematical and Physical Sciences,  La Trobe University, Bendigo 3552, Australia. }
\email{m.hussain@latrobe.edu.au}

\author[N. Shulga]{Nikita Shulga}
\address{Nikita Shulga,  Department of Mathematical and Physical Sciences,  La Trobe University, Bendigo 3552, Australia. }
\email{n.shulga@latrobe.edu.au}
\date{}

\author[H. Takahasi]{Hiroki Takahasi}
\address{Hiroki Takahasi, Department of Mathematics, Keio University, Yokohama, 223-8522, Japan.}
\email{hiroki@math.keio.ac.jp}

%{37C05, 37D25}
\maketitle
    \begin{abstract}
%Let $\bff:=(f_n)_{n\geq 1}$ be an infinite iterated function system on $[0,1]$ satisfying some natural regularity properties, and denote its attractor by $\Lambda(\bff)$. For any infinite subset $D\subseteq \mathbb{N}$ define
%\[
%\E(\bff,D):=\{ x \in \Lambda(\bff): a_n(x)\in D \text{ for all }n\in\N \text{ and }\lim_{n\to\infty} a_n=\infty\}.
%\]
%We prove the Hausdorff dimension result under the introduction of an additional restriction on a slow growth on digits, i.e. we consider a set
%\[
%S(\bff,D,\varphi):=\left\{ x\in \E(\bff,D) :  a_n(x)\leq \varphi(n) \text{ for all } n\in\N  \right\}
%\]
%for any function $\varphi:\mathbb{N}\to [\min D, \infty)$ such that $\varphi(n)\to\infty$ when $n\to\infty$. We show that its Hausdorff dimension stays the same no matter how slowly the function $\varphi$ grows.
%We show that its Hausdorff dimension stays the same no matter how slowly the function $\varphi$ grows. One of the consequences of our result is the recent work of Takahasi (2023), which only dealt with regular continued fraction expansions. 
%We also extend our result to the slow growth of a product of $m\in\N$, not necessarily consecutive digits.
%We also extend our result to slowly growing products of (not necessarily consecutive) digits.
%\textcolor{blue}{(My opinion. In the abstract, definitions should be avoided (unnecessary). Math symbols should be as least as possible.) 
For an infinite iterated function system $\bff$ on $[0,1]$ with an  attractor $\Lambda(\bff)$ and for an infinite subset $D\subseteq \mathbb{N}$, consider the set
\[
\E(\bff,D)=
\{ x \in \Lambda(\bff): a_n(x)\in D \text{ for all }n\in\N \text{ and }\lim_{n\to\infty} a_n=\infty\}.
\]
For a function $\varphi:\mathbb{N}\to [\min D, \infty)$ such that $\varphi(n)\to\infty$ as $n\to\infty$,
we compute the Hausdorff dimension of the set
\[
S(\bff,D,\varphi)
=
\left\{ x\in \E(\bff,D) :  a_n(x)\leq \varphi(n) \text{ for all } n\in\N  \right\}.
\]
We prove that the Hausdorff dimension
stays the same no matter how slowly the function $\varphi$ grows.
One of the consequences of our result is the recent work of Takahasi (2023), which only dealt with regular continued fraction expansions. 
%We also extend our result to the slow growth of a product of $m\in\N$, not necessarily consecutive digits.
We further extend our result to slowly growing products of (not necessarily consecutive) digits.  

    \end{abstract}

% \author[M. Hussain]{Mumtaz Hussain}
% \address{Mumtaz Hussain,  Department of Mathematical and Physical Sciences,  La Trobe University, Bendigo 3552, Australia. }
% \email{m.hussain@latrobe.edu.au}

% \author[N. Shulga]{Nikita Shulga}
% \address{Nikita Shulga,  Department of Mathematical and Physical Sciences,  La Trobe University, Bendigo 3552, Australia. }
% \email{n.shulga@latrobe.edu.au}

% {\color{red} \textbf{Suggested abstract.} Let $\bff:=(f_n)_{n\geq 1}$ be an infinite iterated function system on $[0,1]$ with $\boldsymbol{\xi}$-regularity and the bounded distortion property, and $\Lambda(\bff)$ its attractor. For any infinite subset $D\subseteq \mathbb{N}$ define the set
% \[
% \E(\bff,D)
% :=
% \{ x \in \Lambda(\bff): a_n(x)\in D \text{ for all }n\in\N \text{ and }\lim_{n\to\infty} a_n=\infty\}.
% \]
% For each function $\varphi:\mathbb{N}\to [\min D, \infty)$ such that $\varphi(n)\to\infty$ when $n\to\infty$, we compute the Hausdorff dimension of the set 
% \[
% S(\bff,D,\varphi)
% :=
% \left\{ x\in \E(\bff,D) :  a_n(x)\leq \varphi(n) \text{ for all } n\in\N  \right\}.
% \]
% A consequence of our result is the recent work Takahasi (2023), the fourth named author. However, our proofs do not involve ergodic theoretic tools for the Gauss map and the mass distribution principle used by Takahasi, and rely on an explicit construction of Cantor sets. We extend our results to the product of $m\in\N$ digits.} 

\section{Introduction}

%\textcolor{blue}{The current format is unfriendly to my eyes, and probably the same to editors and referees. Can you enlarge the font size and put more side margins?}

%\commumtaz{In my opinion, our results should be stated within the first two pages of the paper. One way to do it is to give a brief introduction, then the general setup and statement of our results, then as corollaries we may list previous significant contributions. If required, in the next subsection, we expand upon the literature and quote necessary results}

The theory of infinite Iterated Function Systems on the unit interval includes important number systems such as the regular continued fractions, whose relevance cannot be understated. Recall that 
each irrational $x\in(0,1)$
   has a unique
 regular continued fraction expansion 
 $x=
 1/(a_1+1/(a_2+\cdots))$, where each partial quotient $a_n=a_n(x)$, $n\geq1$ belongs to the set $\mathbb N$ of positive integers.
A pioneering result in the fractal dimension theory of continued fraction is due to 
 %Among classical results in this direction,
Jarn\'ik who proved that the set of irrationals in $(0,1)$ with bounded partial quotients 
is of Hausdorff dimension $1$.
Good \cite{Good1941} proved that the set of irrationals in $(0,1)$ whose partial quotients tend to infinity
%\[\left\{x\in\mathbb I\colon \lim_{n\to\infty}a_n(x)=\infty\right\}\]
is of Hausdorff dimension $1/2$.
Various extensions and refinements of Good's theorem have been made which determine fractal dimensions of sets of irrationals whose partial quotients obey certain
restrictions or growth conditions, %prescribed conditions,
see %\cite{CaoWangWu2013,Cus90,FLWW,Hir70,Hirst1973,JaeKes10,MR2869111,Luc97,Moo92,Ram85,MR4607611,WangWu08hirst} 
\cite{CaoWangWu2013,Cus90,FLWW,Hir70,MR2869111,Luc97,Moo92,MR4607611,WangWu08hirst} 
for example.

Hirst \cite{Hirst1973} proved results analogous to that of Good in the case when $a_n$'s belong to some subset of $\mathbb N$.
% For every infinite subset $D\subseteq \N$, the convergence exponent of $D$  is 
%\[
%\tau(D)
%:=
%\inf\left\{ s\geq 0: \sum_{n\in D} n^{-s}<\infty\right\}.
%\]
For every infinite subset $D\subseteq \N$,
 Hirst considered the set
\[
E(D)
:=
\left\{x\in(0,1)\setminus\mathbb Q\colon a_n(x)\in D \text{ for all } n\in \N  \text{ and } \lim_{n\to\infty} a_n(x)=\infty \right\}.
\]
He introduced the exponent of convergence 
\[
\tau(D)
:=
\inf\left\{ s\geq 0: \sum_{n\in D} n^{-s}<\infty\right\},
\]
%He introduced a number $\tau(D)$ called the exponent of convergence (see \eqref{exponent} and the remark after it for the definition), 
and proved
%\textcolor{red}{(remove) $\hdim E(D) \leq \tau(D)/2$.}
that the Hausdorff dimension of $E(D)$ does not exceed $\tau(D)/2$.
He also conjectured that $\tau(D)/2$ should be the exact value of the Hausdorff dimension for this set. About thirty five years after the appearance of Hirst's paper, this conjecture was resolved by Wang and Wu in \cite[Theorem~1.1]{WangWu08hirst}.

%He proved the following theorem.
%\begin{theorem}[\cite{Hirst1973}, Theorem 1]\label{thm:Hirst} \begin{equation}     \hdim E(D) \leq \frac{\tau(D)}{2}  . \end{equation} \end{theorem}
%Hirst also conjectured that $\tau(D)/2$ should be the exact value of the Hausdorff dimension for this set. 35 years later than the appearance of Hirst's paper \cite{Hirst1973}, this conjecture was resolved by Wang and Wu in \cite{WangWu08hirst}.
%where they proved
%\begin{theorem}[\cite{WangWu08hirst}, Theorem 1.1]\label{thm:WangWu} \begin{equation}         \hdim E(D) = \frac{\tau(D)}{2} . \end{equation} \end{theorem}
Naturally, one might wonder that if additional restrictions are introduced on the growth rate of $a_n$, the Hausdorff dimension may drop. In this regard, there are two natural sets to consider, 
 \[\{x\in E(D)\colon a_n(x)\geq \varphi(n)\text{ for all } n\in\mathbb N\},\]
 %and
 \[\left\{x\in E(D)\colon a_n(x)\leq \varphi(n)\text{ for all }n\in\mathbb N\right\},\]
 where $\varphi\colon\mathbb{N}\rightarrow \mathbb N$ is a function satisfying $\lim_{n\to \infty}\varphi(n)=\infty$. 
   Interestingly, the %Hausdorff 
   dimension of the first set can drop from $\tau(D)/2$ and even become $0$ when $\varphi$ grows very rapidly \cite{CaoWangWu2013,Cus90,FLWW,MR2869111,Luc97,Moo92,WanWu08II} while
  the %Hausdorff
  dimension of 
  the second set does not drop no matter how slowly $\varphi$ grows, as precisely stated below. Let $\hdim$ denote the Hausdorff dimension. 
  \begin{theorem}[\cite{MR4607611}, Theorem~1.1]\label{thm:Takahasi}
Let $D\subseteq \N$ be an infinite subset and $\varphi\colon\N\to [\min D,\infty)$ such that $\varphi(n)\to\infty$ as $n \to\infty$. Then, 
\[
\hdim\left\{ x\in E(D): a_n(x)\leq \varphi(n) \text{ for all } n\in\N\right\}
=
\frac{\tau(D)}{2}.
\]
\end{theorem}

  This `no dimension drop' 
  seems to be a ubiquitous phenomenon, and generalizations of
   the above theorem
   to diverse settings are desired.
    The first generalization of Theorem \ref{thm:Takahasi} was done by Nakajima and the last-named author in \cite{NakajimaTakahasi2023} to the semi-regular continued fractions, viewed as a non-autonomous iterated function system \cite{RU16}. 
   In this paper, we generalize Theorem~\ref{thm:Takahasi} in two directions while providing a much simpler proof.% than the one restricted  only to regular continued fractions in \cite{MR4607611}. %\textcolor{blue}{(Add) with a much simpler proof than \cite{MR4607611}}. 

 First, we treat a question of the slow growth of each digit in an infinite Iterated Function System (iIFS) (Theorem~\ref{thm:onedigit}).
   Second, we extend this result to the weighted products of an arbitrary finite number of digits in the given iIFS
(Theorem~\ref{thm:product}).
  Below we introduce a definition of an iIFS and state the main results.
%  In this paper we extend a result of Theorem~\ref{thm:Takahasi} to  any $\boldsymbol{\xi}$-regular iIFS $\bff$.   Moreover, we show that the same result would hold if we bound a weighted product of arbitrary number of digits of the given iIFS, instead of bounding just each one of the digits.

\subsection{Infinite Iterated Function Systems}
Consider a sequence $\bff=(f_n)_{n\geq1}$ of functions $f_n : [0,1] \to [0,1]$, satisfying 
\begin{enumerate}[label={\roman*}.]
\item \textit{\bf Smoothness Property:} $f_n\in C^1([0,1])$ for each $n\geq1$;

\item \textit{\bf Contraction Property:} There exists an integer $m$ and a real number $0 < \rho < 1$, such that for any $(a_1,\ldots,a_m)\in \N^m$ and $x\in[0,1]$,
\[
0< | ( f_{a_1} \circ \cdots \circ f_{a_m})^\prime (x) | \leq \rho <1;
\]

 \item \textit{\bf Open Set Condition:} For any $i\neq j\in\N$, we have $f_i((0,1)) \cap f_j((0,1)) = \origemptyset.$   
\end{enumerate}
The system $\bff=(f_n)_{n\geq1}$ is called an \textit{infinite iterated function system} (iIFS). 
By the smoothness and the contraction properties, we can define a natural projection $\Pi_{\bff}: \N^\N \to [0,1]$ as
\[
\Pi_{\bff}(\bfa) = \lim_{n\to\infty} f_{a_1} \circ \cdots \circ f_{a_n} (1),
\]
where $\bfa=(a_n)_{n\geq1} \in \N^\N$. Also, we denote by $\Lambda(\bff)$ the attractor of the iIFS $(f_n)_{n\geq1}$, that is,
\[
\Lambda(\bff) = \Pi_{\bff} (\N^\N).
\]

%----------(removable?)--------
%The smoothness and the contraction property of an iIFS guarantee the existence of the limit defining $\Pi_{\bff}$.  The name \emph{attractor} reminisces us of finite iterated function systems. When we consider finitely many contractions $f_1,\ldots, f_n$, there is a unique compact set $K$ satisfying  
%\[
%K=\bigcup_{j=1}^n f_j(K).
%\]
%This set would be the unique fixed point, hence an attractor, of a contraction function defined on the family of compact subsets of $[0,1]$ to itself \cite[Chapter 9]{falconer_book2014}. ---------------

By the open set condition, 
%except for countably many points in $\Lambda$, to each $x \in \Lambda$, 
to all but countably many points $x$ in $\Lambda$
we can associate a unique sequence $(a_n)_{n\geq1}$ of integers such that
\[
x = \lim_{n\to\infty} f_{a_1} \circ \cdots \circ f_{a_n} (1).
\]
For a given $x$ we refer to the sequence $(a_n)_{n\geq1}= (a_n(x))_{n\geq1}$ as a \textit{sequence of digits} of $x$.

We say that an iIFS $\bff=(f_n)_{n\geq 1}$ is $\boldsymbol{\xi}$\textit{-regular} if in addition it satisfies 
\begin{enumerate}[label={\roman*}.]
\stepcounter{enumi}\stepcounter{enumi}\stepcounter{enumi}
\item \textit{\bf Regularity Property}: There exists a sequence of positive numbers $\boldsymbol{\xi}=(\xi_n)_{n\geq 1}$ such that for every $\varepsilon>0$ there are positive constants $c_1,c_2$ depending on $\varepsilon$ such that 
\[
\frac{c_1}{\xi_{n}^{1 + \varepsilon}} \leq |f'_n(x)| \leq \frac{c_2}{\xi_{n}^{1 - \varepsilon}}
\quad
\text{ for all }  n\in\N 
\text{ and  all }  x\in [0,1].
\]
\end{enumerate}
If $\xi_n = n^d$, the $\boldsymbol{\xi}$-regular iIFS $\bff$ is called $d$-decaying system, as defined by Jordan and Rams in \cite{MR2869111}. Finally, we assume that iIFS satisfies the following condition.
\begin{enumerate}[label={\roman*}.]
\stepcounter{enumi}\stepcounter{enumi}\stepcounter{enumi}\stepcounter{enumi}
\item \textit{\bf Bounded Distortion Property} (BDP): There exists a constant $\kappa\geq1$ such that for every $n\in\N$ and $a_1,\ldots, a_n \in \N$ we have 
\[
\left|(f_{a_1}\circ\ldots \circ f_{a_n})'(x) \right|
\leq
\kappa \left|(f_{a_1}\circ\ldots \circ f_{a_n})'(y) \right|
\text{ for all }x,y\in [0,1].
\]
\end{enumerate}
We note that all natural examples of IFS satisfy BDP, see Section \ref{sec:examples} below.

%In this paper we extend a result of Theorem~\ref{thm:Takahasi} to any $\boldsymbol{\xi}$-regular iIFS $\bff$. Or, alternatively, we introduce a restriction on a slow growth of digits for the set from Theorem \ref{thm:CWW}. Moreover, we show that the same result would hold if we bound a weighted product of arbitrary $m$ digits of the given iIFS, instead of bounding just each one of the digits.

\subsection{Statements of results}
Let %$\bff=(f_n)_{n\geq 1}$ 
$\bff$ be a $\boldsymbol{\xi}$-regular iIFS.
%with the Bounded Distortion Property. %and $D\subseteq \N$. 
For an infinite subset $D\subseteq \N$ 
%and any sequence of positive numbers $\boldsymbol{\xi}=(\xi_n)_{n\geq 1}$ 
define 
\[
\E(\bff,D):=\left\{ x\in \Lambda(\bff): a_n(x)\in D \text{ for all } n\in\N, \lim_{n\to\infty} a_n(x)=\infty  \right\},
\]
and
\begin{equation}\label{exponent}
s_0(D;\boldsymbol{\xi})
:=
\inf\left\{ s\geq 0: \sum_{n\in D} \frac{1}{\xi_n^s} <\infty\right\}.
\end{equation}
For $2$-decaying systems such as the regular continued fractions, we have $\tau(D)=2s_0(D;\boldsymbol{\xi})$.
 In \cite{CaoWangWu2013},
% \cite[Theorem~1.1]{CaoWangWu2013}, 
 the solution of Hirst's conjecture was generalized to digits in an iIFS in the following way:
    If $\bff$ satisfies the BDP, then 
   $\hdim \E(\bff,D)=s_0(D;\boldsymbol{\xi}).$

For a function $\varphi:\N\to [\min D,\infty)$,
%such that $\varphi(n)\to\infty$ when $n\to\infty$,
define 
% \[
% S(\bff,D,\varphi):=\left\{ x\in \Lambda(\bff): a_n(x)\in D, \, a_n(x)\leq \varphi(n) \text{ for all }  n\in\N, \text{ and }   \lim_{n\to\infty} a_n(x)=\infty  \right\}
% \]
\[
S(\bff,D,\varphi):=\left\{ x\in \E(\bff,D) :  a_n(x)\leq \varphi(n) \text{ for all } n\in\N  \right\}.
\]
The first main result of the present paper is the following statement.

\begin{theorem}\label{thm:onedigit}
Let $\bff$ be a $\boldsymbol{\xi}$-regular iIFS with the BDP. For every infinite subset $D\subseteq \N$ and
 any function $\varphi:\N\to [\min D,\infty)$ such that $\varphi(n)\to\infty$ as $n\to\infty$,
we have 
\[
\hdim S(\bff,D,\varphi) = s_0(D;\boldsymbol{\xi}).
\]
\end{theorem}

%\textcolor{red}{The motivation for considering products of digits should be explained, citing Hussain...}
Theorem~\ref{thm:onedigit} concerns the behavior of a single digit.
 Kleinbock and Wadleigh \cite{KleinbockWadleigh} proved that the product of consecutive partial quotients in regular continued fractions is closely related to improvements of Dirichlet's Theorem in Diophantine approximation. Such sets have been further studied in \cite{HuangWuXu2020,HKWW, HussainShulgaExponential} for example. 
 Hence, it is natural to generalise Theorem~\ref{thm:onedigit} to include products of digits. 

%Using the result of 
Using Theorem~\ref{thm:onedigit} we derive a much more general result involving weighted products of a fixed number of digits. 
%To be more precise, 
For a fixed $m\in\N$ consider $\overline{t} = (t_1,\ldots,t_{m})\in \R_{>0}^m$ and a vector $\overline{g} = (g_1,\ldots,g_m)$ of arbitrary functions $g_i\colon\N\ \to \N, 1\leq i \leq m$. Define a set
\[
P_m(\bff,D,\varphi, \overline{t}, \overline{g}):= \left\{ x\in \E(\bff,D) :  a_{g_1(n)}^{t_1}(x)\cdots a_{g_m(n)}^{t_{m}}(x)\leq \varphi(n) \text{ for all } n\in\N  \right\}
\]
and denote $M(D):=\max\{\min D, (\min D)^{m \max_i t_i} \}$. We have the following theorem.
%Our second result is a statement on the Hausdorff dimension of this set.
\begin{theorem}
\label{thm:product}Let $\bff$, $D$ be as in Theorem~\ref{thm:onedigit} and let $\varphi: \N \to [M(D),\infty)$ be such that $\varphi(n)\to\infty$ as $n\to\infty$. For any $\overline{t}\in \R_{>0}^m$ and for any functions $g_i\colon \N \to \N, 1\leq i \leq m$, we have
\[
\hdim P_m(\bff,D,\varphi, \overline{t}, \overline{g}) = \hdim S(\bff,D,\varphi) = s_0(D;\boldsymbol{\xi}).
\]
\end{theorem}

%The result of this theorem
Theorem~\ref{thm:product} means that in fact 
the Hausdorff dimension of $P_m(\bff,D,\varphi, \overline{t}, \overline{g})$ does not depend on $m,\overline{t}$ and functions $g_i$'s no matter how they behave.
In particular, if we let $g_i(n)=n+i$ and $\overline{t}=(1,\ldots,1)$, we get a set of slowly growing products of consecutive digits 
%of the $\boldsymbol{\xi}$-regular iIFS
\[
\left\{ x\in \E(\bff,D) :  a_{n+1}(x)\cdots a_{n+m}(x)\leq \varphi(n) \text{ for all } n\in\N  \right\}.
\]
If we let $g_i(n) = ni$ and $\overline{t}=(1,\ldots,1)$, we get a set of slowly growing products of digits, located in an arithmetic progression %$\boldsymbol{\xi}$-regular iIFS
\[
\left\{ x\in \E(\bff,D) :  a_{n}(x)\cdots a_{nm}(x)\leq \varphi(n) \text{ for all } n\in\N  \right\}.
\]
%\textcolor{red}{
The set of numbers with large products of digits in arithmetic progressions (that is, numbers $x\in[0,1]$ with $a_{n}(x)\cdots a_{nm}(x)\geq \varphi(n)$ for infinitely many $n$) for the case of the regular continued fractions was studied in \cite{HussainShulgaFiniteProgressions}.%}

\subsection{Examples of infinite iterated function systems}\label{sec:examples}
%\textcolor{green}{Hiroki: For a better organization, I think it might be better to put the list of examples to the end of section 3.}
%\textcolor{red}{Nikita: I think the examples are more suitable here, as we have just stated the results OK(Hiroki)}
We present several examples of $\boldsymbol{\xi}$-regular iIFS to which our main results apply.
%\begin{enumerate}
    %\item 

\subsubsection{$2$-decaying systems}
    \begin{itemize}
    \item {\bf Regular Continued fractions.} For each $n\in\N$ define %\textcolor{red}{(remove) $f_n:[0,1)\to [0,1)$}
    %\textcolor{blue}{(replace) $f_n:[0,1]\to [0,1]$}
    $f_n:[0,1]\to [0,1]$ by
$$
f_n(x) = \frac{1}{x+n}, \quad x\in [0,1].
$$
This iIFS is induced by the Gauss map $T: [0,1]\to[0,1]$ defined by $T(0)=0$ and 
$$T(x) = \frac{1}{x} \mod\, 1 \quad\text{ for } x\neq0.$$
The contraction constants are $m=2$ and $\rho=\frac{1}{2}$ and the BDP can be shown as in \cite[Lemma~2.4]{NakajimaTakahasi2023}.
\item {\bf L\"uroth expansions.}  For each $n\in\N$ define $f_n:[0,1]\to [0,1]$ by
    %\textcolor{blue}{(replace) $f_n:[0,1]\to [0,1]$} by
$$
f_n(x) = \frac{x}{n(n+1)} + \frac{1}{n+1}, \quad x\in(0,1], n\in\N.
$$
This system is induced by the L\"{u}roth map $T: (0,1]\to (0,1]$ defined by
\begin{equation}
\label{lurothmap}
T(x)=d_1(x)(d_1(x)-1)\left(x-\frac{1}{d_1(x)}\right),\ \  \text{where}\hspace{0.3cm} d_1(x)=\left[\frac{1}{x}\right]+1.
\end{equation}
% \item Even Integer Continued Fractions (see, for example, \cite[Chapter 3]{Schwieger1995} ).  The iIFS $\bff=(f_n)_{n \geq 1}$ given by the maps
% \[
% f_n(x)
% :=
% \frac{1}{-x + n+1} \text{ for odd } n\in \N \text{ and all } x\in [0,1],
% \]
% and 
% \[
% f_n(x)
% :=
% \frac{1}{x + n} \text{ for even } n\in \N \text{ and all } x\in [0,1],
% \]
% is induced by the map $T:[0,1]\to [0,1]$ given by
% \[
% T(x)
% =
% \begin{cases}
% x^{-1} - 2k, &\text{ if } (2k+1)^{-1} \leq x \leq (2k)^{-1}, k\in\N,\\
% 2k - x^{-1}, &\text{ if } (2k)^{-1}\leq x \leq (2k-1)^{-1}, k\in\N
% \end{cases}
% \]
% \item Odd-Odd Continued Fractions \cite{KimLeeLiao2022}. The iIFS $\bff=(f_n)_{n \geq 1}$ given by the maps

% \[
% f_n(x)
% :=
% \frac{(n+1) x + (n-1)}{(n+3)x + n+1} \text{ for odd } n\in \N \text{ and all }
% x\in [0,1],
% \]
% and
% \[
% f_n(x)
% :=
% \displaystyle \frac{(n-2)x + n}{nx + n+2} \text{ for even } n\in \N \text{ and all }
% x\in [0,1].
% \]
% is induced by the map $T:[0,1]\to [0,1]$ given by
% \[
% T(x)
% :=
% \begin{cases}
% \frac{kx - (k-1)}{k - (k+1)x}, &\text{ if } 1 - \frac{2}{2k} \leq x \leq 1 - \frac{2}{2k+1}, \, k\in \N, \\
% \frac{k - (k+1)x}{kx - (k-1)}, &\text{ if } 1 - \frac{2}{2k + 1} \leq x \leq 1 - \frac{2}{2 k+ 2}, \, k\in \N,\\
% 1, &\text{ if } x=1.
% \end{cases}
% \]

The contraction constants are $m=1$ and $\rho=\frac{1}{2}$ and the BDP holds for $\kappa=1$. \textsc{Figure}~\ref{Fig:GMLM} shows the graphs of the Gauss and the Lüroth maps.
%\item tbd

%\begin{figure}[ht!]
\begin{figure}[ht!]
\begin{center}
\includegraphics[scale=0.75,  trim={5.0cm 16cm 8.5cm 4.0cm},clip]{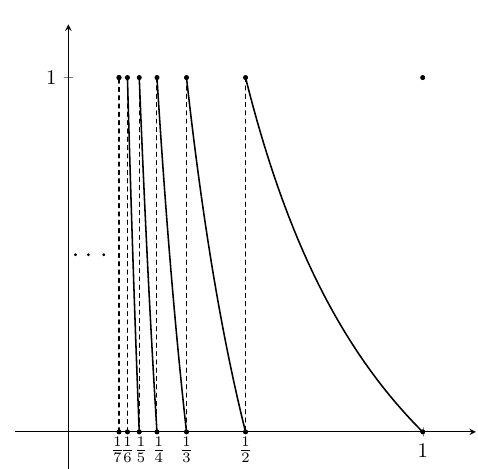}
\includegraphics[scale=0.75,  trim={5.0cm 16cm 8.5cm 4.0cm},clip]{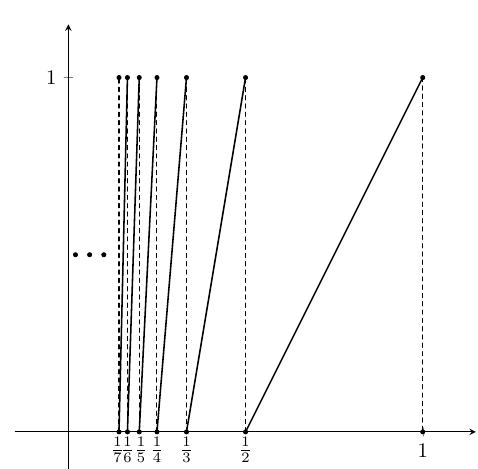}
\caption{ Left: the Gauss map. Right: the L\"uroth map.}\label{Fig:GMLM}
\end{center}
\end{figure}
\end{itemize}

%\item $d$-decaying systems for $d\neq2$.
\subsubsection{$d$-decaying systems for $d\neq2$}
\begin{itemize}
    \item {\bf Quadratic Gauss map.} This is a particular case of $f$-expansion (as defined by R\'enyi \cite{Renyi1952}) which is a $3$-decaying iIFS. Lett $p,q\in \R$ be such that $p\geq -1$ and $q>0$. Consider the polynomial $P(x)=x^2 +px + q$ and define $f:[0,\infty)\to (0,1]$ by $f(x) = \frac{q}{P(x) + x}$. Then the map is defined as 
    $$T(x) = f^{-1}(x) \mod\, 1.$$

    In particular, if we let $p=q=1$, then $f(x) = \frac{1}{x^2+2x+1}=\frac{1}{(x+1)^2}$, the map is
$$
T(x) = \frac{1}{\sqrt{x}} \mod\, 1,
$$
(see \textsc{Figure}~\ref{Fig:QGM}) and we have
$$
f_n(x) = \frac{1}{(x+n)^2}, \quad x\in[0,1], n\in\N.
$$
    
    %The parameters of the Contraction Property are $m=2$ and $\rho=\frac{1}{2}$. Again, using the Contraction Property and that the derivatives $f_n'$ are uniformly bounded and arguing as in \cite[Lemma~2.4]{NakajimaTakahasi2023}, we can show the BDP. We refer the reader to Chapters 3 and 13 in \cite{Schwieger1995} for more details. The map $T$ is depicted in \textsc{Figure}~\ref{Fig:QGM}.
    The parameters of the Contraction Property are $m=2$ and $\rho=\frac{1}{2}$. Again, we can show the BDP as in \cite[Lemma~2.4]{NakajimaTakahasi2023}. We refer the reader to 
    %Chapters 3 and 13 in 
    \cite[Chapters~3 \& 13]{Schwieger1995} for more details. 
\begin{figure}[ht!]
\begin{center}
\includegraphics[scale=0.75,  trim={5.0cm 15.4cm 8.5cm 4.0cm},clip]{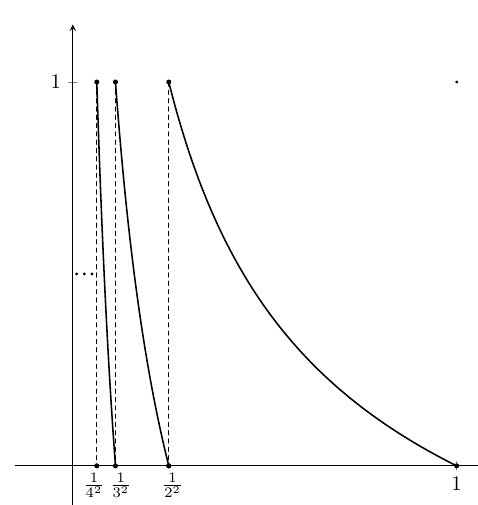}
\includegraphics[scale=0.75,  trim={5.0cm 16cm 9.0cm 4.0cm},clip]{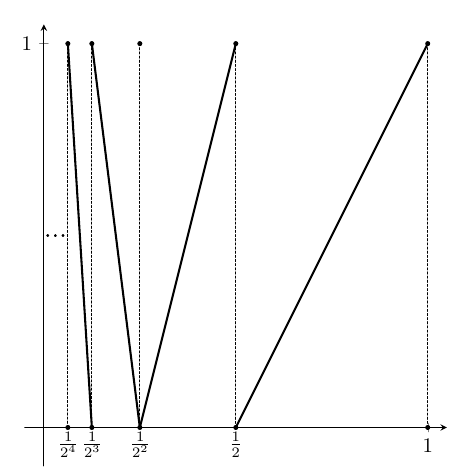}
\caption{Left: The Quadratic Gauss map. Right: The Generalised L\"uroth Series.}\label{Fig:QGM}
\end{center}
\end{figure}

\end{itemize}
        
%\item $\boldsymbol{\xi}$-regular iIFSs, which are not $d$-decaying for any $d$.
\subsubsection{
$\boldsymbol{\xi}$-regular iIFSs, which are not $d$-decaying for any $d$.}

\begin{itemize}
    \item {\bf Generalised L\"uroth series (GLS)} was introduced by Barrionuevo et al. in  \cite{BarrionuevoBurtonDajaniKraaikamp1996}. GLS may have finitely or infinitely many digits. In the latter case, they can be interpreted as a $\boldsymbol{\xi}$-regular iIFS.  Let $\varepsilon=(\varepsilon_j)_{j\geq 1} $ be any sequence in $\{0,1\}$, and let $\{I_n=(l_n,r_n]:n\in\N\}$ be a collection of pairwise disjoint intervals such that
\[
\sum_{n\in \N} (r_n - l_n)=1 
\quad\text{ and}\quad
(r_n - l_n)_{n\geq 1} \text{ is non-increasing}.
\]

Write $I_{\infty}:=[0,1]\setminus \bigcup_{n\in \N} I_n$ and define the maps $T,S,\varepsilon:[0,1]\to [0,1]$ by
\[
T(x)
=
\begin{cases}
\frac{x - l_n}{r_n- l_n}, &\text{ if } x\in I_n,\\
0, &\text{ if } x\in I_{\infty}, 
\end{cases}
\quad
S(x)
=
\begin{cases}
\frac{r_n - x}{r_n- l_n}, &\text{ if } x\in I_{n}, \\
0, &\text{ if } x\in I_{\infty},
\end{cases}
\]
\[
\varepsilon(x):=
\begin{cases}
\varepsilon_n, &\text{ if } x\in I_{n},  \\
0, &\text{ if } x\in I_{\infty}.
\end{cases}
\]
The \textit{Generalised Lüroth map} $T_{\varepsilon}:[0,1]\to [0,1]$ is given by
\[
T_{\varepsilon}(x) = \varepsilon(x) S(x) + (1-\varepsilon(x))T(\varepsilon). 
\]
The map $T_{\varepsilon}$ induces the $\xi$-regular iIFS $\bff=(f_n)_{n\geq 1}$ given by
\[
f_n(x)
=
\begin{cases}
- (r_n-l_n)x + r_n, &\text{ if } \varepsilon_n=1, \\
(r_n-l_n)x - l_n, &\text{ if } \varepsilon_n=0.
\end{cases}
\]
The parameters of the Contraction Property are $m=1$ and $\rho=r_1-l_1$. As for the BDP, we have $\kappa=1$. 
%Clearly, we recover the $2$-decaying iIFS for L\"uroth series when $\varepsilon=(\varepsilon_n)_{n\geq 1}$ is given by $\varepsilon_n=0$ for all $n\in\N$ and we consider the intervals $\{((n+1)^{-1},n^{-1}]:n\in\N\}$. 
A non $d$-decaying $\boldsymbol{\xi}$-regular iIFS is obtained when we consider $\{[2^{-n}, 2^{-n+1}):n\in\N\}$ and any sequence $\varepsilon$. \textsc{Figure}~\ref{Fig:QGM} shows the map $T_{\varepsilon}$ for $\varepsilon=(1,1,-1,-1,\ldots)$.
%\begin{figure}[ht!]
% \begin{figure}
% \begin{center}
% \includegraphics[scale=0.75,  trim={5.0cm 16cm 9.0cm 4.0cm},clip]{FIG-GLS.pdf}
% \caption{ The Generalised L\"uroth Series.}\label{Fig:GLS}
% \end{center}
% \end{figure}

The definition of the GLS in \cite{BarrionuevoBurtonDajaniKraaikamp1996} considers a proper infinite subset $B$ of $\N$ rather than $\N$.
%\textcolor{red}{(remove) an infinite subset $B$ of $\N$ rather than $\N$.}
%\textcolor{blue}{(replace) 
%}
However, we may assume $B=\N$ after enlisting the elements of $B$ in an increasing manner.

%    \item tbd
    
\end{itemize}

\subsection{Outline of proofs of the theorems}
The rest of this paper consists of two sections, entirely dedicated to proofs of Theorems~\ref{thm:onedigit} and
\ref{thm:product}. In the proofs, the upper bounds of the Hausdorff dimension are rather trivial as they can be obtained from the known results. The issue is to obtain the lower bounds.
%In this paper we generalize \cite[Theorem~1.1]{MR4607611}   to (autonomous) iterated function systems, with a much simpler proof.  
A derivation of lower bounds on the Hausdorff dimension of fractal sets often relies on the mass distribution principle \cite{falconer_book2014},
   and the construction of a mass usually becomes intricate depending on the structure of the fractal set under consideration. In the proof %of Theorem \ref{thm:Takahasi} in 
   \cite[Theorem~1.1]{MR4607611}, the construction of masses on fractal sets heavily relies on the ergodic theory of the Gauss map generating the continued fraction.
   Our idea is to dispense with the mass distribution principle altogether, and instead explicitly construct a well-organized Cantor set whose Hausdorff dimension can easily be bounded from below. Starting from a sufficiently large $n$, the compact sets in each layer of our Cantor set correspond to certain admissible digits. The lower estimate of its Hausdorff dimension will be a consequence of two properties: the grandchildren of any compact set appearing in any layer are well separated and the digits defining the layers tend to infinity at a very slow pace.

To show the lower bound in
Theorem~\ref{thm:product}, we construct several subsets of the set under consideration, eventually showing that for some specific choice of a growth function, the set from Theorem~\ref{thm:onedigit} is a subset of a set from Theorem~\ref{thm:product}. The upper bound follows trivially by lifting the restriction on the growth of products.

%\textcolor{red}{Can you briefly explain an idea for a proof of Theorem~\ref{thm:onedigit}?}

\section{Lower bound of Hausdorff dimension of certain Cantor subsets %Preliminaries
}\label{Section:Preliminaries}

In this section, we prove a lower bound for the Hausdorff dimension of certain Cantor sets. Similar arguments date back to Jarn\'ik in 1928 \cite{Jarnik1} 
%\textcolor{red}{reference missing?} 
and they have been used in Cao-Wang-Wu \cite[Theorem~1.1]{CaoWangWu2013} among others much more recently. Our setup is inspired by Kleinbock-Weiss \cite[Section~2.4]{KleinbockWeiss2010}. 

%Denote the \textit{diameter} of a non-empty set $A\subseteq [0,1]$ by $|A| := \sup\{ |x - y|:x,y\in A\}$. %For two non-empty sets $A,B\subseteq [0,1]$, we consider
%\[d(A,B):=\inf\left\{ |a-b|:a\in A, b\in B\right\}.\]
Denote the diameter of a non-empty set $A\subseteq [0,1]$ by $|A|$. Let $\{\calA_n:n\in\N_0\}$ be a collection of finite families of pairwise disjoint compact subintervals of $[0,1]$ with a positive diameter such that $\calA_0$ has exactly one element and the following properties hold: 
\begin{itemize}
\item[•] For all $n\in\N$ and $B\in \calA_n$, there is some $A\in \calA_{n-1}$ such that $B\subseteq A$.
\item[•] For all $n\in\N$ and $A\in \calA_{n-1}$, there are at least two sets $B\in \calA_n$ such that $B\subseteq A$.
\item[•] We have $\delta_n:=\max\{|A|:A\in\calA_n\}\to 0$ as $n\to\infty$.
\end{itemize}
For $n\in\N$ and $A\in \calA_n$, write $\Des_n(A):=\{B\in \calA_{n+1}: B\subseteq A\}$. Under additional considerations, we provide a lower estimate for the Hausdorff dimension of the set
\[
\mathbf{A}_{\infty}
:=
\bigcap_{n=0}^{\infty} \bigcup_{A\in \calA_n} A.
\]
\begin{lemma}\label{LemHD:LowerBound}
Assume there is some $\varepsilon>0$ such that for all $n\in \N$, $A\in \calA_n$, and every $Y,Z\in \Des_{n}(A)$ with $Y\neq Z$ we have
\begin{equation}\label{Eq:LemHD:LowerBound:01}
\min\left\{ |y-z|: y\in\bigcup_{Y'\in \Des_{n+1}(Y)} Y', z \in  \bigcup_{Z'\in \Des_{n+1}(Z)} Z'\right\}
\geq |A|^{1+\varepsilon}.
\end{equation}
%\[d\left(\bigcup_{Y'\in \Des(Y)} Y',\bigcup_{Z'\in \Des(Z)} Z'\right)\geq |A|^{1+\varepsilon}.\]

If $s>0$ satisfies
\begin{equation}\label{Eq:LemHD:LowerBound:02}
\sum_{B\in \Des_{n}(A)} |B|^{s(1+\varepsilon)} \geq |A|^{s(1+\varepsilon)}
\quad\text{ for all } \quad
A\in \bigcup_{n\in\N}\calA_n,
\end{equation}
then $\hdim \mathbf{A}_{\infty}\geq s$.
\end{lemma}
\begin{proof}
Take $\delta>0$ and let $\mathcal{G}$ be a $\delta$-covering of $\mathbf{A}_{\infty}$. Without loss of generality, we may assume $\mathcal{G}$ is finite, consists of open intervals (see, for example, \cite[Section 3.4]{falconer_book2014}), and every element in $\mathcal{G}$ intersects $\mathbf{A}_{\infty}$. For each $G\in \mathcal{G}$ define
\[
n(G)
:=
\max\left\{ n\in\N_0: \#\{ A\in\calA_n: G\cap A\neq \origemptyset \}=1 \right\}.
\]
Let us explain why $n(G)$ is well-defined. First, writing $\calA_0=\{K_0\}$, we have $G\cap K_0\neq \origemptyset$. Second, take $x\in G\cap \mathbf{A}_{\infty}$ and $r>0$ such that $(x-r,x+r)\subseteq G$. Since $\delta_n\to 0$ as $n\to\infty$, for large $n\in\N$, if $A_n\in\calA_n$ contains $x$, then $A_n\subseteq G$, and thus $B\subseteq G$ for all $B\in\Des_n(A_n)$. Therefore, $n(G)$ is the maximum of a non-empty finite set of integers. Let $A(G)$ be the unique compact interval for which the maximum in the definition of $n(G)$ is attained. Then there are two different sets $Y,Z\in \Des_{n(G)}(A(G))$ intersecting $G$. Pick $Y'\in \Des_{n(G)+1}(Y)$ and $Z'\in \Des_{n(G)+1}(Z)$ intersecting $G$. Then,  \eqref{Eq:LemHD:LowerBound:02} yields
\[
|G|
\geq
\min\{|y-z|:y\in Y',z\in Z'\}
\geq 
|A(G)|^{1+\varepsilon}.
\]
% \[
% n(G)
% :=
% \max\left\{ n\in\N: \text{ there exists } A \in \calA_n \text{ such that } \mathbf{A}_{\infty}\cap G\subseteq A \right\}.
% \]
% Clearly, $n(G)$ exists since $\mathbf{A}_{\infty}\cap G$ is contained in the single element of $\calA_0$ and $\delta_n\to 0$ as $n\to \infty$. Also, the set $A(G)$ witnessing the definition of $n(G)$ is unique and there are two different sets $Y,Z\in \Des_{n(G)}(A(G))$ intersecting $\mathbf{A}_{\infty}\cap G$. Pick $Y'\in \Des_{n(G)+1}(Y)$ and $Z'\in \Des_{n(G)+1}(Z)$ intersecting $\mathbf{A}_{\infty}\cap G$. Then,  \eqref{Eq:LemHD:LowerBound:02} yields
% \[
% |G|
% \geq
% |\mathbf{A}_{\infty}\cap G|
% \geq 
% \min\{|y-z|:y\in Y',z\in Z'\}
% \geq 
% |A(G)|^{1+\varepsilon}.
% \]
As a consequence, we have 
\[
\sum_{G\in \mathcal{G}} |A(G)|^{s(1+\varepsilon)}
\leq
\sum_{G\in \mathcal{G}} |G|^{s}.
\]
Clearly, $\mathcal{G}_1:=\{A(G):G\in \mathcal{G}\}$ also covers $\mathbf{A}_{\infty}$. Moreover, since every two members in $\mathcal{G}_1$ are either disjoint or one is contained in the other, there is a subcover $\mathcal{G}_2\subseteq \mathcal{G}_1$ of $\mathbf{A}_{\infty}$ consisting of pairwise disjoint sets. Let $N:=\max\{ n\in \N: \mathcal{G}_2\cap \calA_n \neq \origemptyset\}$. Pick $B\in \mathcal{G}_2\cap \calA_N$ and $A\in \calA_{N-1}$ such that $B\subseteq A$. Since the members of $\mathcal{G}_2$ are pairwise disjoint, $\mathcal{G}_2$ covers $\bfA_{\infty}$, and we cannot go any level deeper than $N$, we must have $\Des_{N-1}(A)\subseteq \mathcal{G}_2$. Hence, if $\mathcal{G}_3$ is the cover obtained from $\mathcal{G}_2$ by replacing the sets in $\Des_{N-1}(A)$ with $A$, we deduce from \eqref{Eq:LemHD:LowerBound:02} that
\[
\sum_{C\in \mathcal{G}_3} |C|^{s(1+\varepsilon)}
\leq 
\sum_{C\in \mathcal{G}_2} |C|^{s(1+\varepsilon)}.
\]
After finitely many replacements,  we arrive at $\calA_0=\{K_0\}$ and thus
\[
0
<
|K_0|^{s(1+\varepsilon)}
\leq 
\sum_{G\in \mathcal{G}} |G|^{s}.
\]
Therefore, we obtain $\hdim\mathbf{A}_{\infty}\geq s$.
\end{proof}

\section{Proofs of the main results}
In this section, we prove Theorem~\ref{thm:onedigit} and Theorem~\ref{thm:product}.
%In this section we complete the proof of Theorem~\ref{thm:onedigit} and provide a proof of Theorem~\ref{thm:product}.
%\textcolor{blue}{(Add?) Finally we present ^several examples of $\boldsymbol{\xi}$-regular iIFS to which our main results apply.}
\subsection{Proof of Theorem~\ref{thm:onedigit}}
\label{Section:ProofOfOneDigit}
We separate the proof into two parts: the upper bound and the lower bound. The upper bound trivially follows from Theorem~1.1 in \cite{CaoWangWu2013}, because $S(\bff,D,\varphi) \subseteq \E(\bff,D) $, and so
\[
\hdim S(\bff,D,\varphi) \leq \hdim \E(\bff,D) = s_0(D;\boldsymbol{\xi}).
\]
Now we show the lower bound of $\hdim S(\bff,D,\varphi)$, that is
%Now we want to show the lower bound of the Hausdorff dimension of $S(\bff,D,\varphi)$, so we want to show
\[
\hdim S(\bff,D,\varphi) \geq s_0(D;\boldsymbol{\xi}).
\]
For this purpose, we will construct a suitable Cantor subset $K(\bff,D,\bfl,\bfu)\subseteq S(\bff,D,\varphi)$, where $\bfl=(l_n)_{n\geq1}$ and $\bfu=(u_n)_{n\geq1}$ are two slowly increasing sequences of natural numbers to be defined later. 

There is nothing to prove if $s_0(D;\boldsymbol{\xi}) =0$, as we always have the trivial lower bound
$\hdim S(\bff,D,\varphi) \geq 0.$
Hence, we further assume that $s_0(D;\boldsymbol{\xi})>0$. Given $n\in\N$ and $\veca=(a_1,\ldots, a_n)\in \N^n$, we consider 
\[
I_n(\veca)
:=
\{ x\in \Lambda: a_k(x)=a_k \text{ for all }k\in\{1,\ldots, n\}\}.
\]
We refer to sets of the form $I_n(\veca)$ as \textit{fundamental intervals}. For readability, let us additionally assume $m=1$ in the Contraction Property. Then, for all $a\in\N$ and all $x\in [0,1]$ we have $|f_a'(x)|\leq\rho$, so
\[
|I_n(\veca)|\leq \rho^n 
\quad 
\text{ for all } n\in \N
 \text{ and all } \veca\in \N^n.
\]
At the end of the proof, we will explain how to proceed when $m\geq2$. 

Let $s,\varepsilon>0$ satisfy $s(1+\varepsilon)^2< s_0(D;\boldsymbol{\xi})$ and let $c_1=c_1(\varepsilon)>0$ and $c_2=c_2(\varepsilon)>0$ be as in the definition of the $\boldsymbol{\xi}$-regularity. Let $\kappa\geq1$ be the constant from the Bounded Distortion property. Write $D=\{d_1<d_2<d_3<\cdots\}$.

We now construct two sequences $\bfl$ and $\bfu$ in $\N$. First, in the view of $s(1+\varepsilon)< s_0(D,\boldsymbol{\xi})$, we may take a strictly increasing sequence $(r_k)_{k\geq 1}$ in $\N$ such that 
\begin{equation}\label{Eq:DefBn:00}
r_k - k\geq 4\quad\text{ and }\quad
\sum_{n=k}^{r_k} \left(\frac{ \kappa^{-1}  c_1}{\xi_{d_n}^{1+\varepsilon}}\right)^{s(1+\varepsilon)}\geq 3 \quad\text{ for all } k\in\N.
\end{equation}
Let $(N_j)_{j\geq 1}$ be a strictly increasing sequence of natural numbers such that for every $j\in\N$ we have
\begin{align}
\varphi(n) &\geq d_{r_j} + 1 \quad \text{ for all } n\in\N_{\geq N_j},\;\text{ and } \label{Eq:DefBn:01} \\
\max &\left\{ (\xi_{d_i}\xi_{d_k})^{1+ \varepsilon}: j \leq i \leq r_{j+1},\ j \leq k \leq r_{j+2}\right\} \leq  \kappa^{-1}  c_1^2\rho^{-\varepsilon N_j}. \label{Eq:DefBn:02}
\end{align}
% Let $\bfl=(l_n)_{n\geq 1}$ and $\bfu=(u_n)_{n\geq 1}$ be given by
% \[
% l_n:=
% \begin{cases}
% 1, & 1\leq n< N_1, \\
% k, & N_k\leq n< N_{k+1},
% \end{cases} 
% \quad
% u_n:=
% \begin{cases}
% 1, & 1\leq n< N_1, \\
% r_k, & N_k\leq n< N_{k+1},
% \end{cases} 
% \]
Let $\bfl=(l_n)_{n\geq 1}$ and $\bfu=(u_n)_{n\geq 1}$ be defined as follows: for each $n\in\N$, take $j\in\N$ such that 
\[
N_j-N_1+1\leq n\leq N_{j+1}-N_1
\]
and put
\[
l_n:= j 
\quad\text{ and }\quad
u_n:= r_j. 
\]
Moreover, for each $n\in\N$ define the set
\[
D_n:= \{d_i\in D\colon l_{n}\leq i< u_{n}\}.
\]
We now describe the layers of a Cantor set contained in $S(\bff,D,\varphi)$. First, consider
\[
K_0:=\overline{I}_{N_1-1}(d_1,\ldots, d_1)
\]
Here and in the following, the bar over $I$ denotes its closure. For all $n\in\N$ and $(b_1,\ldots, b_n)\in\N^n$ define
\[
K_n(b_1,\ldots, b_n)
:=
\overline{I}_{N_1-1+n}(d_1,\ldots, d_1, b_1,\ldots, b_n).
\]
Define $B_1:=D_1$. For each $b\in B_1$, let $L(b)\in D_2$ (resp. $R(b)\in D_2$) be such that $K_2(b,L(b))$ (resp. $K_2(b,R(b))$) is the leftmost (resp. rightmost) interval in $\{K_2(b,b_2):b_2\in D_2\}$ and write
\[
B_2(b):=\left\{(b,b_2):b_2\in D_1\setminus\{L(b),R(b)\}\right\}.
\]
Define
\[
B_2:=\bigcup_{b_1\in B_1} B_2(b_1).
\]
In what follows, for any $j,k\in\N$ and any $\veca=(a_1,\ldots, a_j)\in\N^j$ and $\vecb=(b_1,\ldots, b_k)\in\N^k$, we write 
\[
\veca\vecb:=(a_1,\ldots, a_j,b_1,\ldots,b_k).
\]
Let us assume that for some $n\in\N$ we have defined a non-empty finite set $B_n\subseteq D_1\times \cdots\times D_n$. For each $\vecb \in B_n$, let $L(\vecb)\in D_{n+1}$ (resp. $R(\vecb)\in D_{n+1}$) be such that $K_{n+1}(\vecb L(\vecb))$ (resp. $K_{n+1}(\vecb R(\vecb))$) is the leftmost (resp. rightmost) interval in $\left\{K_{n+1}(\vecb b):b\in D_{n+1} \right\}$ and write
\[
B_{n+1}(\vecb)
:=
\left\{ \vecb b:b\in D_{n+1}\setminus \{L(\vecb),R(\vecb)\}\right\}.
\]
Define
\[
B_{n+1}
:=
\bigcup_{\vecb\in B_n} B_{n+1}(\vecb).
\]
Note that each set $B_n$ depends on $D_n$, which in turn is defined in terms of $\bfl$ and $\bfu$, and hence of $(N_j)_{j\geq 1}$. We will apply Lemma~\ref{LemHD:LowerBound} to estimate the Hausdorff dimension of
\[
K
=
K(\bff,D,\bfl,\bfu)
:=
\bigcap_{n\in\N}\bigcup_{\vecb\in B_n} K_n(\vecb)
\subseteq
S(\bff,D,\varphi).
\]
Let us check that $K$ is indeed a subset of $S(\bff,D,\varphi)$. Pick any $x\in K$. It suffices to check the growth and the limit conditions, for the definition of $K$ implies that every digit of $x$ belongs to $D$. Take $n\in\N$. If $1\leq n\leq N_1-1$, then
\[
a_n(x) = d_1 \leq \min_{k\in\N} \varphi(k) \leq \varphi(n).
\]
If $n\geq N_1$, let $j\in\N$ be the single number for which $N_j\leq n \leq N_{j+1}-1$. Then we have $N_j - N_1+1\leq n-N_1+1 \leq N_{j+1}-N_1$, and so $l_{n-N_1+1}=j$, $u_{n-N_1+1}=r_j$. This yields
\[
a_n(x) \in D_{n-N_1+1}= \{d_j< \cdots < d_{r_j}\}. 
\]
As a result, by \eqref{Eq:DefBn:01} and the choice of $j$ we obtain
\[
d_j \leq a_n(x) \leq d_{r_j} \leq \varphi(n).
\]
Finally, as $n\to\infty$ we have $j\to \infty$, $d_j\to\infty$ and so $a_n(x)\to\infty$. We conclude $x\in S(\bff,D,\varphi)$.

% {\color{blue} ----------------- Probably we'll erase this.--------------------}
% We now describe the layers of a Cantor set contained in $S(\bff,D,\varphi)$. First, consider
% \[
% K_0:=\overline{I}_{N_1-1}(d_1,\ldots, d_1)
% \]
% Here and in the following, the bar above $I$ denotes the closure. For all $n\in\N$ and $(a_1,\ldots, a_n)\in\N^n$ define
% \[
% K_n(a_1,\ldots, a_n)
% :=
% \overline{I}_{N_1-1+n}(d_1,\ldots, d_1, a_1,\ldots, a_n).
% \]
% For $n\in\N_0$, put %\textcolor{red}{(remove, it is hard to see the triple subscripts) put $D_n:=D\cap [d_{l_{N_1}-1+n} , d_{u_{N_1}-1+n})$.}
% \[
% D_n:= \{d_i\in D\colon l_{N_1-1+n}\leq i\leq  u_{N_1-1+n}\}.
% \]
% {\color{blue} ----------------- Probably we'll erase this.--------------------}

% {\color{blue}
% Observe that, given $j\in\N$, for any $n\in\N$ such that $N_j-N_1 +1 \leq n \leq N_{j+1}-N_1$ we have 
% \[
% N_j \leq n + N_1 - 1 \leq N_{j+1} - 1,
% \]
% so $l_n= j$ and $u_n=r_j$. This means that $D_{N_j - N_1+1} = D_{N_j - N_1+2} = \ldots =  D_{N_{j+1} - N_1}$. 
% }
%\textcolor{red}{The definition of $D_n$ is not easy because the subscripts are shifted unnecessarily. Why not the following?
%\[
%D_n:= \{d_i\in D\colon l_{n}\leq i< u_{n}\},
%\]
%where
%\[l_n=k\ (N_k-N_1+1\leq n\leq N_{k+1}-N_1), k=1,2,\ldots\]
%\[u_n=r_k\ (N_k-N_1+1\leq n\leq N_{k+1}-N_1), k=1,2,\ldots\]}

%\textcolor{red}{Does $N_j$, $j\geq2$ really appear in the definition of this set? I only see $N_1$.}
Define $\calA_0:=\{K_0\}$ and $\calA_n:=\{ K_n(\vecb):\vecb\in B_n\}$ for all $n\in\N$. Let us verify the assumptions on $\{\calA_n:n\in\N_0\}$ from Section \ref{Section:Preliminaries}. Take $n\in\N$. Clearly, $\calA_n$ consists of finitely many pairwise disjoint compact intervals with positive diameter. The three points assumed on $\calA_n$ also hold:
\begin{itemize}
\item[•] The first point follows from the construction of $B_n\subseteq D_1\times\cdots\times D_n$.
\item[•] By $r_n-n\geq 4$, for any $A\in \calA_n$ the set $\Des_{n}(A)$ has at least two elements.
\item[•]Since we have assumed $m=1$ in the Contraction Property, we have $\delta_n\leq \rho^{N_1-1}\rho^n$.
%\textcolor{red}{$d_n$ is an element of $D\subset\N$. Do you mean $|K_n(a_1,\ldots, a_n)|\leq \rho^{N_1-1}\rho^n$?}
%\comgero{Fixed. $d_n$ here was referring to the maximum of diameters, but it was certainly ambiguous. Now it is $\delta_n$.}
\end{itemize}
To check \eqref{Eq:LemHD:LowerBound:01} from Lemma~\ref{LemHD:LowerBound}, take $j\in\N$ such that $N_j-N_1+1\leq n\leq N_{j+1}-N_1$, and $\vecb\in B_n$. First, note that if $b_{n+1} \in D_{n+1}$ is such that $\vecb b_{n+1}\in B_{n+1}$ then for all $b_{n+2} \in D_{n+2}$ we have
\begin{align*}
|K_{n+2}(\vecb b_{n+1} b_{n+2})|
&\geq
|K_{n}(\vecb)|\frac{ \kappa^{-1}  c_1^2}{(\xi_{b_{n+1} }\xi_{b_{n+2} })^{1+\varepsilon}} \\
&\geq
|K_{n}(\vecb)| \rho^{\varepsilon N_j} &&\text{(by \eqref{Eq:DefBn:02})} \\
& \geq
|K_{n}(\vecb)|\rho^{\varepsilon n}\\
&\geq
|K_{n}(\vecb)|^{1+\varepsilon}.
\end{align*}
%\textcolor{red}{How the first inequality is deduced? The mean value theorem gives $\exists\theta_1,\theta_2\in[0,1]$ s.t. $|K_{n+2}(\veca a_{n+1} a_{n+2})|=|(f_{a_1}\circ\cdots \circ f_{a_{n+2}})'(\theta_1)|=|(f_{a_1}\circ\cdots \circ f_{a_{n}})'(f_{a_{n+1}}\circ f_{a_{n+2}}(\theta_1))||(f_{a_{n+1}} \circ f_{a_{n+2}})'(\theta_1)|$ and $|K_{n}(\veca  )|=|(f_{a_1}\circ\cdots \circ f_{a_{n}})'(\theta_2)|$. The first equality gives $|K_{n+2}(\veca a_{n+1} a_{n+2})|\geq|(f_{a_1}\circ\cdots \circ f_{a_{n}})'(f_{a_{n+1}}\circ f_{a_{n+2}}(\theta_1))|\frac{c_1^2}{(\xi_{a_{n+1} }\xi_{a_{n+2} })^{1+\varepsilon}}$. For the desired inequality we want $|(f_{a_1}\circ\cdots \circ f_{a_{n}})'(f_{a_{n+1}}\circ f_{a_{n+2}}(\theta_1))|\geq |K_{n}(\veca  )|=|(f_{a_1}\circ\cdots \circ f_{a_{n}})'(\theta_2)|$, but this is not known ... I think we need to assume that the distortion  of $(f_{a_1}\circ\cdots \circ f_{a_{n}})$ is uniformly bounded.}
Let $b,c\in D_{n+1}$ and $b', c'\in D_{n+2}$ be such that $\vecb b b',\vecb cc' \in B_{n+2}$. Then the construction of $B_{n+2}$ implies that the distance between $K_{n+2}(\vecb b b')$ and $K_{n+2}(\vecb c c')$ is at least the length of one of the removed intervals:
\[
\left\{ K_{n+2}(\vecd): \vecd\in \{\vecb b L(\vecb b), \vecb b R(\vecb b), \vecb c L(\vecb c), \vecb c R(\vecb c) \}\right\}.
\]
This length is larger than or equal to $|K_n(\vecb)|^{1+\varepsilon}$, which shows \eqref{Eq:LemHD:LowerBound:01} from Lemma \ref{LemHD:LowerBound}. 

Condition \eqref{Eq:LemHD:LowerBound:02} holds because the Bounded Distortion Property and  \eqref{Eq:DefBn:00} imply 
\[\begin{split}
\sum_{d\in B_{n+1}(\vecb)} |K_{n+1}(\vecb d)|^{s(1+\varepsilon)} 
&\geq 
|K_{n}(\vecb)|^{s(1+\varepsilon)} \left( \sum_{d\in D_{n+1}} \left( \frac{\kappa^{-1} c_1}{\xi_{d}^{1+\varepsilon}} \right)^{s(1+\varepsilon)} - 2 \right) \\
&\geq |K_{n}(\vecb)|^{s(1+\varepsilon)}.
\end{split}\]
%where for the second inequality we used \eqref{Eq:DefBn:00}.
%\textcolor{red}{The same problem as above, in deducing the first inequality.}
%Therefore, $\hdim S(\bff,D,\varphi)\geq  \hdim K\geq s(1+\varepsilon)^2$ by Lemma~\ref{LemHD:LowerBound}.
Therefore, by Lemma~\ref{LemHD:LowerBound} we obtain $\hdim S(\bff,D,\varphi)\geq  \hdim K\geq s $ as required.
%\textcolor{red}{why $\hdim K\geq s(1+\varepsilon)^2$? I think $\hdim K\geq s$ by \eqref{Eq:DefBn:00}.} {\color{blue} \textbf{G.} Fixed}
%\textcolor{red}{Why $K\subset S(\bff,D,\varphi)$?} {\color{blue} \textbf{G.} The proof of this statement is in lines 205--211 of the pdf.}

When $m\geq 2$, we add $m$ digits in each layer of our Cantor set construction and adjust the parameters accordingly. More precisely, the sequence $(r_k)_{k\geq 1}$ remains unchanged. 
%for $k\in \N$, we replace \eqref{Eq:DefBn:00} with 
%\[
%\sum_{n=mk}^{r_k} \left(\frac{c_1}{\xi_{d_n}^{1+\varepsilon}}\right)^{s(1+\varepsilon)}\geq 3
%\quad\text{ and }\quad
%r_k - km\geq 4.
%\]
For $(N_j)_{j\geq 1}$, instead of \eqref{Eq:DefBn:01} we assume that for all $j\in\N$ we have $\varphi(n)\geq d_{r_j}+1$ when $n\geq mN_j$  and instead of 
%rather than 
\eqref{Eq:DefBn:02} we impose
\[
\max \left\{ \prod_{k=1}^{2m} \xi_{d_{i_k}}^{1+ \varepsilon}: j \leq i_1 \leq r_{j+1},\ j \leq i_k \leq r_{j+2} \text{ for } 2\leq k \leq  2m \right\} 
\leq 
 \kappa^{-1}c_1^{2m}\rho^{-\varepsilon N_j}.
\]
In order to define the sequences $\bfl=(l_n)_{n\geq 1}$ and $\bfu=(u_n)_{n\geq 1}$, for each $n\in\N$ take $j\in\N$ such that 
\[
m(N_j - N_1) + 1 \leq n \leq m(N_{j+1} - N_1)
\]
and put $l_n:=j$ and $u_n:=r_j$. We start with $K_0:=\overline{I}_{mN_1-1}(d_1,\ldots, d_1)$. Next, for each $\vecb\in B_1:=D_1\times\cdots\times D_m$ we let $L(\vecb)$ and $R(\vecb)$ be the vectors in $D_{m+1}\times \cdots \times D_{2m}$ such that $K_{2m}(\veca L(\vecb))$ and $K_{2m}(\vecb R(\vecb))$ are, respectively, the leftmost and rightmost intervals in the collection $\left\{K_{2m}(\vecb \vecc):\vecc \in D_{m+1}\times \cdots \times D_{2m}\right\}$. Afterwards, we write 
\[
B_2(\vecb):= \left\{ \vecb \vecc: \vecc\in D_{m+1}\times \cdots \times D_{2m}\} \setminus \{ L(\vecb),R(\vecb)\right\}. 
\]
Finally, we define 
\[
B_2:= \bigcup_{\veca\in B_1} B_2(\vecb).
\]
The rest of the argument now closely follows the case $m=1$. 
This completes the proof of Theorem~\ref{thm:onedigit}. \qed

\subsection{ Proof of 
Theorem~\ref{thm:product}}
The proof is once again split into the upper bound and the lower bound.
The upper bound trivially follows from the definition of the set and 
%\textcolor{red}{(remove) Theorem~\ref{thm:CWW}} 
%\textcolor{blue}{(replace) 
\cite[Theorem~1.1]{CaoWangWu2013}, as we always have
\[
P_m(\bff,D,\varphi, \overline{t}, \overline{g}) \subseteq \E(\bff,D),
\]
and so
\[
\hdim P_m(\bff,D,\varphi, \overline{t}, \overline{g}) \leq \hdim \E(\bff,D)=s_0(D;\boldsymbol{\xi}).
\]

As for the lower bound, we will step-by-step construct subsets of the original set $P_m(\bff,D,\varphi, \overline{t}, \overline{g})$ and eventually come to a subset of the form
\[
P_m(\bff,D,\varphi, \overline{t}, \overline{g}) \supseteq S(\bff,D,\zeta)
\]
for a particular choice of a function $\zeta:\N\to[\min D,\infty).$  By the result of Theorem~\ref{thm:onedigit} this will imply Theorem~\ref{thm:product}.

First, without loss of generality we can assume that $\varphi$ is strictly increasing, as for any function $\varphi:\N\to [\min D,\infty)$ such that $\varphi(n)\to\infty$ as $n\to\infty$ we can always find another function $\phi$ that is strictly increasing and satisfies $\phi(n) < \varphi(n)$ for any $n\in\N$.
This means we can always pass, if needed, to a subset
$$
P_m(\bff,D,\varphi, \overline{t}, \overline{g})  \supseteq P_m(\bff,D,\phi, \overline{t}, \overline{g}) 
$$
with $\phi$ being a strictly increasing function. 

In what follows we assume that $\varphi$ is strictly increasing. It is easy to see that if one has 
\[
a_{g_i(n)}(x) \leq \varphi(n)^{\frac{1}{mt_i}} \text{ for each } 1\leq i \leq m, \text{ for all } n\in\N,
\]
then 
\[
 a_{g_1(n)}^{t_1}(x)\cdots a_{g_m(n)}^{t_{m}}(x)\leq \varphi(n) \text{ for all } n\in\N.
\]
Hence the first subset we consider is
\[\begin{split}
P_m(\bff,D,\varphi, \overline{t}, \overline{g}) & \supseteq  \left\{ x\in \E(\bff,D) : a_{g_i(n)}(x) \leq \varphi(n)^{\frac{1}{mt_i}}, 1\leq i \leq m, \text{ for all } n\in\N \right\}\\
&=: G_1.
\end{split}\]

Now let  $I_i :=g_i(\N)$.
For each $1\leq i \leq m$ and for each $k\in I_i$, 
% \textcolor{red}{(remove) consider a preimage set $g_i^{-1}(k) := \{ n\in\N : g_i(n) = k \}$ of the number $k$ under the function $g_i.$ 
% For $k\in I_i$, let 
% \[
% M_{i,k} := \min \{ r: r\in g_i^{-1}(k) \}. \] }
 let \[M_{i,k} := \min \{ r\in\mathbb N\colon g_i(r)=k \}.\]
Note that $M_{i,k}$ is well-defined as the minimum in the definition is taken over a non-empty set of natural numbers. The second subset we consider is
\[
G_1 \supseteq \left\{ x\in \E(\bff,D) : a_{r_i}(x) \leq \varphi(M_{i,r_i})^{\frac{1}{mt_i}}, \text{ for all } r_i\in I_i, 1\leq i \leq m  \right\}=: G_2.
\]

For each $1 \leq i \leq m$, define a set $C_i = \N \backslash I_i$ 
%- a complement of the image of the function $g_i$. 
For any $k\in C_i$, let 
\begin{comment}
\begin{eqnarray}
M_i(k) : = \notag
\left\{ \begin{split} 
&\min_{j\in I_i} \{ M_{i,j}: M_{i,j} > k \}, && \quad \text{ if }  \{ M_{i,j},j\in I_i: M_{i,j} > k \}\neq\origemptyset, \\ 
& \max_{j\in I_i} \{ M_{i,j}: M_{i,j} < k \} +k, && \quad \text{ if }  \{ M_{i,j},j\in I_i: M_{i,j} > k \}=\origemptyset.
\end{split}
\right.
\end{eqnarray}
\end{comment}
\begin{eqnarray}
R_{i,k} : = \notag
\left\{ \begin{split} 
&\min\{ \{M_{i,j}\}_{j\in I_i}\cap\mathbb N_{ > k} \}, && \quad \text{ if } \{M_{i,j}\}_{j\in I_i}\cap\mathbb N_{ > k}\neq\origemptyset, \\ 
& \max\{ \{M_{i,j}\}_{j\in I_i}\cap\mathbb N_{ < k} \} +k, && \quad \text{ otherwise.}
\end{split}
\right.
\end{eqnarray}
Note that $R_{i,k}$ is also well-defined as at least one of the sets $\{ M_{i,j}\}_{j\in I_i}\cap\mathbb N_{ > k} $ or $\{ M_{i,j}\}_{j\in I_i}\cap\mathbb N_{ < k}$ is non-empty, because $I_i$ is always non-empty.

For each $1 \leq i \leq m$, define a function $\zeta_i\colon \N \to \N$ by
\begin{eqnarray}
\zeta_i(n) := \notag
\left\{ \begin{split} 
&\varphi(M_{i,n})^{\frac{1}{mt_i}}, && \quad n\in I_i,  \\ 
& \varphi(R_{i,n})^{\frac{1}{mt_i}}, && \quad n\in C_i.
\end{split}
\right.
\end{eqnarray}
It is easy to see that each $\zeta_i(n)$ is non-decreasing and tending to infinity as $n\to\infty.$ Indeed, for each $1\leq i \leq m$ either $I_i$ or $C_i$ is infinite. If $I_i$ is infinite, then $\{M_{i,k}\}_{k\in I_i}$ forms a non-decreasing sequence, tending to infinity as $k\to\infty$. If $C_i$ is infinite, then $\{R_{i,k}\}_{k\in C_i}$ forms a non-decreasing sequence, tending to infinity as $k\to\infty$. As $\varphi$ is a strictly increasing function, we get the claim.

Finally, define a function $\zeta\colon \mathbb N\to[\min D,\infty)$ by
\[
\zeta(n) := \min \{ \zeta_1(n),\ldots, \zeta_m(n) \}.
\]
For any $1\leq i \leq m$ and for any $n\in\N$, we have $\zeta(n) \leq \zeta_i(n)$, and so we can pass to the final subset
\[
G_2 \supseteq \left\{ x\in \E(\bff,D) : a_{n}(x) \leq \zeta(n), \text{ for all } n\in\N  \right\}=: G_3=S(\bff,D,\zeta).
\]
As $\zeta(n)$ is non-decreasing and tending to infinity as $n\to\infty$, by Theorem~\ref{thm:onedigit} we get the desired conclusion. This completes the proof of Theorem~\ref{thm:product}.\qed

\subsection*{Acknowledgments} GGR, MH, and NS were supported by the Australian Research Council discovery project grant number 200100994.
HT was supported by the JSPS KAKENHI 
19K21835, 20H01811. 

\begin{comment}
\textcolor{red}{I cannot usebibliog.bib.}
[Cus90]   Cusick, T.W.: Hausdorff dimension of sets of continued fractions.
  Quart. J. Math. Oxford, {\bf 41}, 277--286 (1990) 

  [Hir70] Hirst, K.E.:
A problem in the fractional dimension
theory of continued fractions.
Quart. J. Math. Oxford, {\bf 21}, 29--35 (1970) 

[JaeKes10]  Jaerisch, J.,  Kesseb\"ohmer, M.: 
 The arithmetic-geometric scaling spectrum for continued fractions.
 Ark. Mat. {\bf 48}, 335--360 (2010) 

[Luc97]
$\L$uczak, T.: On the fractional dimension of sets of continued fractions. Mathematika {\bf 44},
50--53 (1997) 

[Moo92] Moorthy, C.G.: A problem of Good on Hausdorff dimension. Mathematika {\bf 39}, 244--246 (1992) 

[Ram85]  Ramharter, G.: Eine Bemerkung \"uber gewisse Nullmengen von Kettenbr\"uchen.
Ann. Univ. Sci. Budapest. E\"otv\"os Sect. Math.  {\bf 28}, 11--15 (1985) 

[RU16] Rempe-Gillen, L.,  %Urba$\acute{\rm n}$ski
Urba\'nski, M.: Non-autonomous conformal iterated function systems and Moran-set constructions, Trans. Amer. Math. Soc. {\bf 368}, %(3)
1979--2017 (2016)

 [NakTak] Nakajima, Y., Takahasi, H.: Hausdorff dimension of sets with restricted, slowly growing partial quotients in the semi-regular continued fraction, arXiv:2209.08318

 [WanWu08II] Wang, B.-W., Wu, J.: 
A remark on continued fractions with sequences of partial quotients.
J. Number Theory {\bf 128}, 2394--2397 (2008)

\end{comment}

\bibliographystyle{abbrv}
\bibliography{bibliog}

\begin{thebibliography}{10}

\bibitem{BarrionuevoBurtonDajaniKraaikamp1996}
J.~Barrionuevo, R.~M. Burton, K.~Dajani, and C.~Kraaikamp.
\newblock Ergodic properties of generalized {L}\"{u}roth series.
\newblock {\em Acta Arith.}, 74(4):311--327, 1996.

\bibitem{CaoWangWu2013}
C.-Y. Cao, B.-W. Wang, and J.~Wu.
\newblock The growth speed of digits in infinite iterated function systems.
\newblock {\em Studia Math.}, 217(2):139--158, 2013.

\bibitem{Cus90}
T.~W. Cusick.
\newblock Hausdorff dimension of sets of continued fractions.
\newblock {\em Quart. J. Math. Oxford Ser. (2)}, 41(163):277--286, 1990.

\bibitem{falconer_book2014}
K.~Falconer.
\newblock {\em Fractal geometry}.
\newblock John Wiley \& Sons, Ltd., Chichester, third edition, 2014.
\newblock Mathematical foundations and applications.

\bibitem{FLWW}
A.-H. Fan, L.-M. Liao, B.-W. Wang, and J.~Wu.
\newblock On {K}hintchine exponents and {L}yapunov exponents of continued
  fractions.
\newblock {\em Ergodic Theory Dynam. Systems}, 29(1):73--109, 2009.

\bibitem{Good1941}
I.~J. Good.
\newblock The fractional dimensional theory of continued fractions.
\newblock {\em Proc. Cambridge Philos. Soc.}, 37:199--228, 1941.

\bibitem{Hir70}
K.~E. Hirst.
\newblock A problem in the fractional dimension theory of continued fractions.
\newblock {\em Quart. J. Math. Oxford Ser. (2)}, 21:29--35, 1970.

\bibitem{Hirst1973}
K.~E. Hirst.
\newblock Continued fractions with sequences of partial quotients.
\newblock {\em Proc. Amer. Math. Soc.}, 38:221--227, 1973.

\bibitem{HuangWuXu2020}
L.~Huang, J.~Wu, and J.~Xu.
\newblock Metric properties of the product of consecutive partial quotients in
  continued fractions.
\newblock {\em Israel J. Math.}, 238(2):901--943, 2020.

\bibitem{HKWW}
M.~Hussain, D.~Kleinbock, N.~Wadleigh, and B.-W. Wang.
\newblock Hausdorff measure of sets of {D}irichlet non-improvable numbers.
\newblock {\em Mathematika}, 64(2):502--518, 2018.

\bibitem{HussainShulgaExponential}
M.~Hussain and N.~Shulga.
\newblock Metrical properties of exponentially growing partial quotients.
\newblock {\em arXiv:2309.10529}, preprint 2023.

\bibitem{HussainShulgaFiniteProgressions}
M.~Hussain and N.~Shulga.
\newblock Metrical properties of finite product of partial quotients in
  arithmetic progressions.
\newblock {\em arXiv:2309.00826}, preprint 2023.

\bibitem{Jarnik1}
V.~Jarn\'ik.
\newblock Zur metrischen {T}heorie der diophantischen {A}pproximationen.
\newblock {\em Prace mat. fiz.}, 36:91--106, 1928.

\bibitem{MR2869111}
T.~Jordan and M.~Rams.
\newblock Increasing digit subsystems of infinite iterated function systems.
\newblock {\em Proc. Amer. Math. Soc.}, 140(4):1267--1279, 2012.

\bibitem{KleinbockWadleigh}
D.~Kleinbock and N.~Wadleigh.
\newblock A zero-one law for improvements to {D}irichlet's theorem.
\newblock {\em Proc. Amer. Math. Soc.}, 146(5):1833--1844, 2018.

\bibitem{KleinbockWeiss2010}
D.~Kleinbock and B.~Weiss.
\newblock Modified {S}chmidt games and {D}iophantine approximation with
  weights.
\newblock {\em Adv. Math.}, 223(4):1276--1298, 2010.

\bibitem{Luc97}
T.~$\L$uczak.
\newblock On the fractional dimension of sets of continued fractions.
\newblock {\em Mathematika}, 44(1):50--53, 1997.

\bibitem{Moo92}
C.~G. Moorthy.
\newblock A problem of {G}ood on {H}ausdorff dimension.
\newblock {\em Mathematika}, 39(2):244--246, 1992.

\bibitem{NakajimaTakahasi2023}
Y.~Nakajima and H.~Takahasi.
\newblock The dimension theory of semi-regular continued fractions.
\newblock {\em arXiv:2209.08318}, preprint 2023.

\bibitem{RU16}
L.~Rempe-Gillen and M.~Urba\'{n}ski.
\newblock Non-autonomous conformal iterated function systems and {M}oran-set
  constructions.
\newblock {\em Trans. Amer. Math. Soc.}, 368(3):1979--2017, 2016.

\bibitem{Renyi1952}
A.~R\'{e}nyi.
\newblock Representations for real numbers and their ergodic properties.
\newblock {\em Acta Math. Acad. Sci. Hungar.}, 8:477--493, 1957.

\bibitem{Schwieger1995}
F.~Schweiger.
\newblock {\em Ergodic theory of fibred systems and metric number theory}.
\newblock Oxford Science Publications. The Clarendon Press, Oxford University
  Press, New York, 1995.

\bibitem{MR4607611}
H.~Takahasi.
\newblock Hausdorff dimension of sets with restricted, slowly growing partial
  quotients.
\newblock {\em Proc. Amer. Math. Soc.}, 151(9):3645--3653, 2023.

\bibitem{WangWu08hirst}
B.-W. Wang and J.~Wu.
\newblock A problem of {H}irst on continued fractions with sequences of partial
  quotients.
\newblock {\em Bull. Lond. Math. Soc.}, 40(1):18--22, 2008.

\bibitem{WanWu08II}
J.~Wu.
\newblock A remark on continued fractions with sequences of partial quotients.
\newblock {\em J. Number Theory}, 128(8):2394--2397, 2008.

\end{thebibliography}

\end{document}